\documentclass[12pt,a4paper]{article}
\pdfoutput=1

\usepackage[T1]{fontenc}
\usepackage[utf8]{inputenc}
\usepackage[UKenglish]{isodate}
\cleanlookdateon

\usepackage{amsfonts,amsmath,amssymb,amsthm,mathtools} % Symbols & microtypography
\usepackage[usenames,dvipsnames]{xcolor} % Colours with names -- see https://www.overleaf.com/learn/latex/Using_colours_in_LaTeX for list
\usepackage[labelsep=period,labelfont=bf,justification=centering]{caption}
\usepackage{float,graphicx,subcaption}
\usepackage{booktabs,makecell}
\usepackage{enumitem,mdwlist}
\setlist[itemize]{topsep=0ex,itemsep=0ex,parsep=0.4ex}
\setlist[enumerate]{topsep=0ex,itemsep=0ex,parsep=0.4ex}

\usepackage{parskip,fullpage}
\usepackage{standalone}
\usepackage{thm-restate}

\usepackage[hyphens]{url} % Urls together with line breaking them
\usepackage[linktoc=all,hidelinks,colorlinks,unicode=true]{hyperref} % Must be loaded after url
\usepackage[capitalise,compress,nameinlink,noabbrev]{cleveref} % Must be loaded after hyperref
\hypersetup{linkcolor={blue!70!black},citecolor={black},urlcolor={blue!70!black}}

\usepackage{tikz}
\usetikzlibrary{snakes}

\newtheorem{theorem}{Theorem}
\newtheorem{lemma}[theorem]{Lemma}
\newtheorem{claim}[theorem]{Claim}
\newtheorem*{claim*}{Claim}
\newtheorem{observation}[theorem]{Observation}
\newtheorem{conjecture}[theorem]{Conjecture}
\newtheorem{corollary}[theorem]{Corollary}
\newtheorem{question}[theorem]{Question}
\theoremstyle{definition}

\newenvironment{poc}{\begin{proof}}{\end{proof}}

\renewcommand{\ge}{\geqslant}
\renewcommand{\le}{\leqslant}
\renewcommand{\emptyset}{\varnothing}

\newcommand{\defn}[1]{\textcolor{Maroon}{\emph{#1}}}

\newcommand{\mc}[1]{\mathcal{#1}}
\newcommand{\bb}[1]{\mathbb{#1}}

\renewcommand{\v}{\textup{\textsf{v}}}
\newcommand{\e}{\textup{\textsf{e}}}
\renewcommand{\d}{\textup{\textsf{d}}}

\DeclareMathOperator{\comp}{comp}

\DeclarePairedDelimiter{\set}{\{}{\}}

\title{Two Relaxations of the Dominating Hadwiger's  Conjecture}

\date{\today}

\author{
Ant\'onio Gir\~ao\footnotemark[2] \and Sergey Norin\footnotemark[1]\and Youri Tamitegama\footnotemark[4] \and Jane Tan\footnotemark[4]
}

\begin{document}
\maketitle

\renewcommand{\thefootnote}{\fnsymbol{footnote}} % Make affiliation marks symbols

\footnotetext[2]{University College London, UK
(\textsf{\href{mailto:snorine@gmail.com}{a.girao@ucl.ac.uk}}).}

\footnotetext[1]{McGill University, Canada
(\textsf{\href{mailto:snorine@gmail.com}{snorine@gmail.com}}).}

\footnotetext[4]{University of Oxford, UK
(\textsf{\{\href{mailto:tamitegama@maths.ox.ac.uk}{tamitegama},\href{mailto:jane.tan@maths.ox.ac.uk}{jane.tan}\}@maths.ox.ac.uk}).}

\renewcommand{\thefootnote}{\arabic{footnote}} 
\begin{abstract}
Illingworth and Wood recently proposed the Dominating Hadwiger's Conjecture, a strengthening of Hadwiger's Conjecture which asserts that every graph with no dominating $K_t$-model is $(t-1)$-colorable. 
We prove two relaxations of this conjecture. First, we show that every graph with average degree $Ct (\log t)^2$ contains a dominating $K_t$-model for some absolute constant $C$. This bound improves on the $2^{t-2}$ due to  Illingworth and Wood and is within an $O(\log t)$ factor from optimal. Second, we prove that the vertices of every graph with no dominating $K_t$-model can be partitioned into $t-1$ parts such that the subgraph induced by each part has bounded maximum degree.
\end{abstract}

\section{Introduction}

Hadwiger's famous conjecture, posed in 1943~\cite{hadwiger1943klassifikation}, asserts that every graph with no $K_t$-minor is $(t-1)$-colorable. This conjecture and its strengthenings and relaxations have been extensively explored; see Seymour~\cite{seymour2016hadwiger} for a survey. 

Notably, most of the conjectured strengthenings of Hadwiger's conjecture are now known to be false. Most recently, K{\"u}hn, Sauermann, Steiner and Wigderson~\cite{kuhn2025disproof} disproved the Odd Hadwiger's conjecture proposed by Gerards and Seymour in 1993. On the other hand, many natural relaxations of Hadwiger's conjecture are known to hold. For example, following a long history of improving bounds, Delcourt and Postle~\cite{delcourt2025reducing} have shown that  every graph with no $K_t$-minor is $O(t \log\log t)$-colorable, and Dujmovi\'c, Esperet, Morin and Wood~\cite{dujmovic2023clustered} proved the relaxation of Hadwiger's conjecture to clustered colorings.

A new tantalizing strengthening of Hadwiger's conjecture was recently proposed by Illingworth and Wood~\cite{IW25}.
For a graph $G$, a \defn{dominating $K_t$-model} in $G$ is a sequence $(X_1,\dotsc,X_t)$ of pairwise disjoint non-empty subsets of $V(G)$ such that each induced subgraph $G[X_i]$ is connected and every vertex of $X_j$ has a neighbour in $X_i$ whenever $i<j$. Note that replacing ``every vertex of $X_j$'' by ``some vertex of $X_j$'' recovers the standard definition of a \defn{$K_t$-model}, and a graph contains a $K_t$-model if and only if it contains a $K_t$-minor.

\begin{conjecture}[Dominating Hadwiger's Conjecture~\cite{IW25}]\label{c:DomHad}
Every graph with no dominating $K_t$-model is $(t-1)$-colorable.
\end{conjecture}

Illingworth and Wood proved \cref{c:DomHad} for $t \leq 4$, and Gir\~ao et al.~\cite{girao2026dominating} have very recently proved it for $t=5$, thereby generalizing the Four Color Theorem~\cite{appel1989every,robertson1997four}.

In this paper, we prove two relaxations of \cref{c:DomHad} which parallel known relaxations of Hadwiger's conjecture.  

The first concerns the minimum density needed to force a dominating $K_t$-model. 
In the classical, non-dominating setting, 
Kostochka~\cite{kostochka1984lower} and Thomason~\cite{thomason1984extremal} proved that graphs without $K_t$-minors have average degree $O(t\sqrt{\log t})$; constructions of Kostochka and Fern\'andez de la Vega~\cite{fernandez1983maximum} show that this order is best possible. This bound implies that every graph with no $K_t$-minor is $O(t\sqrt{\log t})$-colorable, which was the best known bound until the last decade.
 
Illingworth and Wood~\cite{IW25} showed that the bound for dominating models is different: random graphs with average degree $\Theta(t\log t)$ do not contain dominating $K_t$-models.
They also provided an upper bound of $2^{t-2}$ on the smallest average degree forcing dominating $K_t$-models.
In this paper, we narrow this gap by showing that graphs of average degree $\Omega(t(\log t)^2)$ contain dominating $K_t$-models. In particular, this implies that every graph with no dominating $K_t$-model is $O(t(\log t)^2)$-colorable. 

\begin{restatable}{theorem}{degmain}\label{thm:degmain}
There is an absolute constant $C>0$ such that, for every positive integer $t$, if $G$ is a graph with $\d(G)\ge Ct(\log t)^2$, then $G$ contains a dominating $K_t$-model.
\end{restatable}

Our second main result involves relaxing the notion of proper coloring. Rather than requiring the subgraph induced by the vertex set of a given color to be edgeless, we require it to have bounded maximum degree.

Edwards, Kang, Kim, Oum, and Seymour~\cite{edwards2015relative} proved that the corresponding relaxation of Hadwiger's conjecture holds. Explicitly, they have shown that for every positive integer $t$ there exists $c_t$ such that if a graph $G$ has no $K_t$-minor then there exists a partition $(X_1,\dotsc, X_{t-1})$ of $V(G)$ such that $\Delta(G[X_i])\leq c_t$ for each $i$. Van den Heuvel and Wood~\cite{van2018improper} showed that the defect $c_t$ can be chosen to be $t-2$. We prove a strengthening of the result of~\cite{edwards2015relative} to graphs with no dominating $K_t$-models.

\begin{restatable}{theorem}{deficiency}\label{thm:deficiency}
For every  positive integer $t$ there exists $c_t>0$ such that the vertex set of any graph $G$ that does not contain a dominating $K_t$-model can be partitioned into $t-1$ parts $X_1,\dotsc, X_{t-1}$ with $\Delta(G[X_i])\leq c_t$ for each $i$.
\end{restatable}

In \cref{sec:preliminaries}, we gather some basic notation and tools.
\cref{sec:degree} is dedicated to bounds on average degree that force a dominating $K_t$-model. We give a simple proof of a quadratic bound in \cref{sec:quadratic}; this is weaker than our main result, but the argument also applies to a question of Reed on rooted minors. \cref{sec:main} contains the more technical proof of \cref{thm:degmain}.
\cref{sec:deficiency} then treats \cref{thm:deficiency}, and we conclude with brief remarks in \cref{sec:remarks}.

\section{Preliminaries}\label{sec:preliminaries}
All logarithms are natural, and all graphs we consider are finite and simple.
For a graph $G$, we write $\v(G)=|V(G)|$ , $\e(G)=|E(G)|$, and $\d(G)=2\e(G)/\v(G)$ for its average degree. Let $\deg_G(v)$ (or $\deg(v)$ if $G$ is clear from context) be the degree of a vertex $v \in V(G)$. We use $\Delta(G)$ and $\delta(G)$ to denote the maximum and minimum degree of $G$, respectively. Let $\comp(G)$ denote the number of connected components of $G$.

If $S\subseteq V(G)$ let $N_G(S) = \set{ u \in V(G)\setminus S \: | \:  u\; \mathrm{has \; a \; neighbour \; in}\; S}$ be the (outer) \defn{neighbourhood} of $S$. We  omit the subscript when the ambient graph is clear from context, and write $N_G(v)$ instead  of $N_G(\{v\})$ for $v \in V(G)$. Let $\deg_S(v) = |N_G(v) \cap S|$ for $v \in V(G) \setminus S$, and let $G[S]$ denote the subgraph of $G$ induced by $S$.  

For a pair of disjoint subsets $A,B \subseteq V(G)$, let $G[A,B]$ denote the maximal bipartite subgraph of $G$ with bipartition $(A,B)$, i.e. $V(G[A,B])= A \cup B$ and $E(G[A,B])= \{uv \in E(G) \: | \: u \in A, v \in B\}$. 
Then let $\e_G(A,B)=|E(G[A,B])|$.

For a set $S\subseteq V(G)$, let $G\setminus S$ be the graph obtained from $G$ by removing vertices of $S$. 

 We say that $G$ is \defn{$d$-degenerate} if every non-empty subgraph contains a vertex of degree at most $d$. 

For disjoint $S,T \subseteq V(G)$ we say that $S$ \defn{dominates} $T$ if $T \subseteq N_G(S)$. The following observation will be repeatedly used in our constructions of  dominating models.
\begin{observation}\label{obs:prepend}
	Let $G$ be a graph, let $X \subseteq V(G)$ be non-empty such that $G[X]$ is connected, and let
	$H$ be a subgraph of $G[N(X)]$. If $H$ contains a dominating
	$K_s$-model $(T_1,\ldots,T_s)$, then $(X,T_1,\ldots,T_s)$
	is a dominating $K_{s+1}$-model in $G$.
\end{observation}

In \cref{sec:degree}, we use the following standard Chernoff bounds; see, for example, \cite[Theorem~4.4]{mitzenmacher2017probability}.
\begin{lemma}[Chernoff bounds]\label{lem:chernoff}
If $X$ is binomial with mean $\mu$, then
\[
 \bb{P}(X\leq \mu/2)\leq \exp(-\mu/8)
 \quad\text{and}\quad
 \bb{P}(X\geq 2\mu)\leq \exp(-\mu/3).
\]
The same lower-tail estimate holds for a sum of independent Bernoulli variables with non-identical success probabilities.
\end{lemma}

\section{Density forcing a dominating {$K_t$-model}}\label{sec:degree}

\subsection{A simpler proof of a weaker bound}\label{sec:quadratic}

We begin with a simple proof of a weakening of \cref{thm:degmain} showing that 
average degree at least $6t^2$ guarantees a dominating $K_t$-model. 
 
\begin{lemma}\label{l:degt2} Let $G$ be a connected graph with $\d(G) > 0$. Then for every $v \in V(G)$ there exist disjoint non-empty $S,T \subseteq V(G)$ such that $v \in S$, $G[S]$ is connected,
$T \subseteq N(S)$, and
	\begin{equation} \label{e:degt2}\d(G[T]) \geq \d(G) - 4\sqrt{\d(G)}.	\end{equation} 

\end{lemma}	

\begin{proof}
Let $d =  \d(G)$. If $d \leq 16$ then $S = \{v\}$ and $T = \{u\}$  for any neighbour $u$ of $v$ satisfy the conditions of the lemma. Thus we may assume that $d > 16.$

 We also assume without loss of generality that the conclusion of the lemma holds for every proper connected subgraph $H$ of $G$. This implies the following easy case.

\begin{claim}\label{claim:minimality0} If 
$\d(H) \geq  d$ for some proper subgraph $H$ of $G$ then the lemma holds.
\end{claim}

\begin{proof} We may assume that $H$ is connected by replacing $H$ by a connected component of $H$ with highest average degree. Let $P$ be a path in $G$ with ends $v$ and $v' \in V(H)$, chosen shortest. By our assumption there exist disjoint non-empty $S',T \subseteq V(H)$ such that $v' \in S'$, $G[S']$ is connected,
	$T \subseteq N(S')$, and
$$\d(G[T]) \geq \d(H[T]) \geq\d(H) - 4\sqrt{\d(H)} \geq   d - 4\sqrt{d}.$$
Then $S = S' \cup V(P)$ and $T$ satisfy the conclusion of the lemma for $G$.
\end{proof}

The preceding claim allows us to assume without loss of generality that 
$\e(H) <  \frac{d}{2} \v(H)$ for every proper subgraph $H$ of $G$. Now for a subset $X\subseteq V(G)$, let $f(X)$ denote the number of edges with at least one endpoint in $X$.	
\begin{claim}\label{claim:minimality}
	For any $X\subseteq V(G)$ we have
	\begin{equation} \label{e:minimality}
f(X) \geq \frac{d}{2}|X|.
	\end{equation}
\end{claim}
\begin{proof} If $X = \emptyset$ then the claim trivially holds. 
	If $X=V(G)$, then $f(V(G)) = \e(G) \geq  \frac{d}{2}\v(G)$ by the choice of $d$.
	So let $\emptyset \neq X \subsetneq V(G)$ and let $Y = V(G)\setminus X$. 
	By our earlier assumption we have $\e(G[Y])< \frac{d}{2} |Y|$. It follows that $$ \frac{d}{2} |V(G)| \leq \e(G) = f(X) + \e(G[Y]) \leq f(X) +  \frac{d}{2}|Y|$$ and therefore $f(X) \geq  \frac{d}{2} (|V(G)| - |Y|) =  \frac{d}{2}|X|$.
\end{proof}

Choose disjoint non-empty sets $S,T \subseteq V(G)$ satisfying the first three requirements of the lemma, i.e. $v \in S$, $G[S]$ is connected,
and $T$ is dominated by $S$, such that  $|S|+|T|\geq d^{-1/2}f(S)$ and subject to that  $S$ is maximal.

Let us first note that such a choice is possible. Indeed, let $S = \set{v}$ and let $T = N(v)$. Then $|S| + |T| > |T| = f(S) \geq  d^{-1/2}f(S).$

We claim that $S$ and $T$ satisfy the lemma. Suppose first that some $t \in T$ has at least  $d^{-1/2}\deg_G(t)$ neighbours in $V(G)-S-T$. Let $Z$ be the set of these neighbours, and define $S' = S \cup \set{t}$ and $T'  = (T \setminus \set{t}) \cup Z$. Then $$|S'| + |T'| = |S|+|T|+ |Z| \geq d^{-1/2}f(S) + |Z| \geq d^{-1/2}(f(S)+  \deg_G(t)) \geq   d^{-1/2}f(S'),$$
as  $|Z|\ge d^{-1/2}\deg_G(t)$ and
$f(S')-f(S)\le\deg_G(t)$, but this contradicts the choice of $(S,T)$ to maximize $S$.
It follows that every $t \in T$ has fewer than $d^{-1/2}\deg(t)$ neighbours in $V(G)-S-T$. Thus 
\begin{equation}\label{e:TW}
	\e(G[T,V(G)-S-T]) \leq d^{-1/2} \sum_{t \in T} \deg(t)  \leq 2d^{-1/2} f(T) \leq  2d^{-1/2} f(S\cup T).
\end{equation}
Combining all of the above, we have
\begin{align*}
	\e(G[T]) &= f(S\cup T)- f(S) - \e(G[T,V(G)-S-T]) \\
	&\stackrel{\eqref{e:TW}}{\geq}  f(S\cup T) - d^{1/2}(|S|+|T|) - 2d^{-1/2} f(S \cup T)\\
	&= (1- 2d^{-1/2}) f(S\cup T) -   d^{1/2} (|S|+|T|)\\
	&\stackrel{\eqref{e:minimality}}{\geq} (1- 2d^{-1/2}) \cdot \frac{d}{2}(|S|+|T|) -   d^{1/2} (|S|+|T|)\\ &\geq  \frac{d}{2} (1- 4d^{-1/2})|T|,
\end{align*}
Thus $ \d(G[T]) \geq d - 4\sqrt{d},$
and \eqref{e:degt2} holds.
\end{proof}	

\cref{l:degt2} immediately implies the claimed weakening of \cref{thm:degmain}.

\begin{theorem}\label{thm:degt2}
For any positive integer $t$, if $G$ is a graph with $\d(G)\ge 6t^2$, then $G$ contains a dominating $K_t$-model.
\end{theorem}

\begin{proof}
The proof is by induction on $t$ with the base cases $t=1,2$ being trivial.
For the induction step, we assume without loss of generality that $G$ is connected. Let $S$ and $T$ be as in \cref{l:degt2}. We have
$$\d(G[T]) \geq  \d(G) - 4\sqrt{\d(G)} \geq 6t^2 - 4\sqrt{6}t \geq  6(t-1)^2,$$
	where the last inequality holds for $t \geq 3$. Thus $G[T]$ contains a dominating $K_{t-1}$-model by the induction hypothesis, implying that $G$ contains a dominating $K_{t}$-model by \cref{obs:prepend}.
\end{proof}	

\cref{l:degt2} can also be applied to give a partial answer to the following question of Reed. 

Let $H$ be a graph. A \defn{model} $\mu$ of $H$ in a graph $G$ is a collection $(\mu(w))_{w \in V(H)}$ of pairwise disjoint non-empty subsets of $V(G)$ such that
\begin{itemize}
	\item $G[\mu(w)]$ is connected for every $w \in V(H)$, and
	\item for every $ww' \in E(H)$ some edge of $G$ has one end in $\mu(w)$ and the other in $\mu(w')$.
\end{itemize}
It is easily seen that $G$ contains a model of $H$ if and only if $H$ is a minor of $G$.

Let $\mathrm{c}(H)$ denote the infimum of $d \in \bb{R}_{+}$ such that every graph $G$  with $\d(G) \geq d$ 
contains a model of $H$. Let $\mathrm{Rc}(H)$ (the $R$ is for rooted) be the infimum  
of  $d \in \bb{R}_{+}$ such that for every $w \in V(H)$, every connected graph $G$  with $\d(G) \geq d$ and every $v \in V(G)$, $G$ contains a model $\mu$ of $H$ such that $ v \in \mu(w)$ .

\begin{question}[Reed~\cite{reed2026barbados}]\label{q:Bruce} For which functions $f$ does $\mathrm{Rc}(H)<f(\mathrm{c}(H))$ hold for every graph $H$? In particular, is this true for $f(x)=x+c$ for some constant $c$?
\end{question}

While the second part of \cref{q:Bruce} remains open, 
\cref{l:degt2} implies the following. 

\begin{corollary} For every graph $H$, $\mathrm{Rc}(H) = \mathrm{c}(H) + O(\sqrt{\mathrm{c}(H)})$.
\end{corollary}

\begin{proof}
	It suffices to show that if $d \in \bb{R}_{+}$ satisfies $d - 4\sqrt{d} > \mathrm{c}(H)$ and $d \geq 2$ then for every $w \in V(H)$, every connected graph $G$  with $\d(G) \geq d$, and every $v \in V(G)$, $G$ contains a model $\mu$ of $H$ such that $ v \in \mu(w)$.
	
	Let $S$ and $T$ be as in \cref{l:degt2}. In particular, $\d(G[T]) \geq \d(G) - 4\sqrt{\d(G)} \geq d - 4\sqrt{d} > \mathrm{c}(H)$. Thus $G[T]$ contains a model $\mu$ of $H$. Adding $S$ to $\mu(w)$ we obtain a model of $H$ in $G$ with $v \in \mu(w)$, as required.
\end{proof}

\subsection{A better bound}\label{sec:main}
\newcommand{\proofsection}[1]{\subsubsection{#1}}
\newcommand{\proofsubsection}[1]{\paragraph{#1}}
\newcommand{\maintheoremref}{\cref{thm:degmain}}

We now work toward the proof of \cref{thm:degmain}. Throughout this section, we will use $C,C_0,C_1,\ldots$ to
denote positive absolute constants. Constants without subscripts
may change from line to line, but numbered constants are fixed. 

Ideally, we would like to obtain a strengthening of \cref{l:degt2} with polylogarithmic loss in the degree. This can be done for graphs that are reasonably close to being regular by taking a random dominating subset and then joining together components via an iterative procedure, replacing the greedy process in \cref{l:degt2}. This is accomplished in \cref{sec:almostregular}.

The alternative, when $G$ does not have a suitable almost regular subgraph, requires a much fiddlier process. In this case, using a dichotomy originally established by  
Gir\~ao and Hunter \cite{girao2025induced}, we can assume that $G$ is a bipartite graph in which one side of the bipartition is much larger than the other and consists of vertices of the same degree. This is established in \cref{sec:dichotomy}. We again build the dominating set by starting with a randomly sampled subset and try to join together components, but here we need to carefully maintain sufficient degree and imbalance in each step. 
 The details of this argument are given in \cref{sec:unbalanced}. 
 
 Finally, \cref{sec:final} combines the almost-regular and highly unbalanced cases to derive \cref{thm:degmain}.

\proofsection{Almost regular vs. highly unbalanced dichotomy}\label{sec:dichotomy}

In this section we prove  a slight refinement of the key dichotomy of
Gir\~ao and Hunter \cite[Lemmas~4.2--4.5]{girao2025induced}. Let \(\Gamma\) be a bipartite graph with bipartition \((A,B)\).  We say that \(\Gamma\) is
\(K\)-\defn{almost-biregular} if $\deg_\Gamma(a)\leq \frac{K\e(\Gamma)}{|A|}$ for all $a\in A$, and $\deg_\Gamma(b)\leq \frac{K\e(\Gamma)}{|B|}$ for all $b\in B$. We use verbatim the following lemma from \cite{girao2025induced}, which was in turn borrowed from~\cite{janzer2023Erdos}.

\begin{lemma}[{\cite[Lemma~4.2]{girao2025induced}, \cite[Lemma~3.7]{janzer2023Erdos}}]\label{lem:almost-biregular}
Let \(K\geq 1\), and let \(\Gamma\) be a \(K\)-almost-biregular bipartite graph.  Then \(\Gamma\) has an induced
subgraph \(\Gamma'\) such that
\[
 \d(\Gamma')\geq \frac{\d(\Gamma)}4
 \qquad\text{and}\qquad
 \Delta(\Gamma')\leq 24K\d(\Gamma').
\]
\end{lemma}

Next we need a strengthening of \cite[Lemma 4.3]{girao2025induced} which replaces $\log^2$ factors by single powers of $\log$. This improvement saves a $\log$ factor in \cref{thm:degmain}.

\begin{lemma}\label{lem:regularise}
There exists \(C_1 > 0\) with the following property.
Let \(d\geq 16\) and \(L\geq 2\), and let \(G\) be a graph satisfying
\[
 \d(G)\geq d
 \qquad\text{and}\qquad
 \Delta(G)\leq Ld.
\]
Then \(G\) contains a subgraph \(H\) such that
\[
 \d(H)\geq \frac{d}{C_1\log(2L)}
 \qquad\text{and}\qquad
 \Delta(H)\leq C_1\log(2L)\,\d(H).
\]
\end{lemma}

\begin{proof}
Iteratively deleting vertices of degree less than \(d/2\) does not decrease the average degree below \(d\), so we may assume that
\(\delta(G)\geq d/2\).

Set $r=\left\lceil\log_2(2L)\right\rceil+1$ and define $D_i=2^{i-1}d$.
Partition \(V(G)\) into sets \(V_1,\ldots,V_r\) where $V_i=\{v:D_i/2\leq \deg_G(v)<D_i\}$.
Since \(\sum_{v \in V(G)} \deg_G(v)\geq d\v(G)\), there is some $i$ such that $V_i$ has total degree $s:=\sum_{u\in V_i}\deg_G(u)\geq d\v(G)/r$.

Randomly choose a subset \(F\subseteq V_i\) by putting each vertex of \(V_i\) in \(F\)
independently with probability \(1/2\). Observe that an edge with both ends in \(V_i\)
is counted twice in \(s\) and crosses \((F,V(G)\setminus F)\) with
probability \(1/2\), while an edge with exactly one end in \(V_i\) is counted
once and crosses with probability \(1/2\).  Hence $\bb{E} \left(\e_G(F,V(G)\setminus F)\right)\geq s/4$. 
At the same time, \(\bb{E}\left(|F|\right)=|V_i|/2\leq s/D_i\).  Therefore, some outcome
satisfies
\[
 \e_G(F,V(G)\setminus F)-\frac{D_i}{8}|F|\geq\frac s8.
\]
Fix such an \(F\), and set \(A=F\), \(B=V(G)\setminus F\). Then let
\(\Gamma=G[A,B]\) and let \(m=\e(\Gamma)\).

Let \(q=d/(64r)\), and delete from \(B\) all vertices with fewer than
\(q\) neighbours in \(A\). This deletes less than \(q\v(G)\leq m/8\) edges. We now partition the remaining vertices of \(B\) into buckets according to their
degree to \(A\), defined as $B_j=\{b \in B :q2^j\leq \deg_A(b)<q2^{j+1}\}$. Since \(\deg(b)\leq Ld\) for every $b \in B$, there are at most \(1+\lceil\log_2(64Lr)\rceil\leq4r\) non-empty buckets.
Thus one bucket, say \(B_j\), supports \(\frac{m}{8r}\) edges.  Let
\(\Gamma_j=\Gamma[A,B_j]\), and put \(m_j=\e(\Gamma_j)\). 

We will show that $\Gamma_j$ has a suitable subgraph $H$. Note from the definitions that $m \geq \frac{d\v(G)}{8r}$ and $\frac{m}{|A|} \geq \frac{D_i}{8}$. It follows that $\frac{m_j}{|A|} \geq \frac{D_i}{64r}$ and $\frac{m_j}{|B_j|} \geq q = \frac{d}{64r}$. Recalling that $A\subset V_i$, we have $\max_{a\in A}\deg_{\Gamma_j}(a)\leq D_i
 \leq 64r\,\frac{m_j}{|A|}$. Also, using the definition of $B_j$, we have  $\max_{b\in B_j}\deg_{\Gamma_j}(b)
 <2q2^j\leq 2\,\frac{m_j}{|B_j|}$.
Thus, \(\Gamma_j\) is \(64r\)-almost-biregular.  Furthermore,
\[
 \d(\Gamma_j)
 =\frac{2m_j}{|A|+|B_j|}
 \geq
 \min\left\{\frac{m_j}{|A|},\frac{m_j}{|B_j|}\right\}
 \geq\frac{d}{64r}.
\]
Applying \cref{lem:almost-biregular} gives a subgraph \(H\) with $\d(H)\geq\frac{d}{256r}$ and $\Delta(H)\leq 1536r\,\d(H)$.
Since \(r=\Theta(\log(2L))\), this proves the lemma.
\end{proof}

We now state the key dichotomy that will separate the almost regular and highly unbalanced cases in the proof.
It is analogous to \cite[Lemma 4.5]{girao2025induced}, but uses \cref{lem:regularise} in place of  \cite[Lemma 4.3]{girao2025induced}. We include a proof for completeness.

\begin{lemma}\label{lem:dichotomy}
There exists \(C_2 > 0\) such that the following holds.
Let \(d\geq 32\), \(L\geq 32\), and let \(G\) be an \(n\)-vertex
\(d\)-degenerate graph with \(\d(G)\geq d\).  Then at least one of the
following occurs:
\begin{enumerate}[label=\textup{(\roman*)}]
 \item\label{item:dichotomy-unbalanced}
 there is a partition \(V(G)=A\cup B\) such that $|A|\geq \frac L4|B|$ and $\e_G(A,B)\geq\frac{nd}{8}$;
 \item\label{item:dichotomy-regular}
 \(G\) contains a subgraph \(H\) such that $\d(H)\geq\frac{d}{C_2\log L}$ and $\Delta(H)\leq C_2\log L\cdot \d(H)$.
\end{enumerate}
\end{lemma}

\begin{proof}
Let $B=\{v\in V(G):\deg_G(v)\geq Ld\}$ and let $A=V(G)\setminus B$.
Since a \(d\)-degenerate $n$-vertex graph has at most \(dn\) edges, we have $|B|Ld\leq\sum_{v\in V(G)}\deg_G(v)\leq 2dn$, and so \(|B|\leq 2n/L\). Hence, \(|A|\geq (L/4)|B|\).

If \(\e_G(A,B)\geq nd/8\), then \ref{item:dichotomy-unbalanced} holds and we are done.
Otherwise, degeneracy gives $\e(G[B])\leq d|B|\leq \frac{2dn}{L}\leq\frac{dn}{16}$. It follows that $\e(G[A])
 \geq\frac{dn}{2}-\frac{dn}{8}-\frac{dn}{16}
 \geq\frac{dn}{4}$. 
Thus \(\d(G[A])\geq d/2\) while
\(\Delta(G[A])<Ld\), so applying \cref{lem:regularise} to \(G[A]\) with
parameters \(d/2\) and \(2L\) gives us
\ref{item:dichotomy-regular}.
\end{proof}

\proofsection{Connected neighbourhoods in almost-regular graphs}\label{sec:almostregular}

Our next step is to establish \cref{thm:degmain} in the case when outcome (ii) of \cref{lem:dichotomy} holds. The main technical step in this part of the argument is \cref{lem:low-cost}, which can be viewed as an analogue of \cref{l:degt2}.

We start with an easy and well-known observation.

\begin{lemma}\label{lem:connect-dominating}
Let \(G\) be connected and let \(D\) be a non-empty dominating set in
\(G\).  If \(G[D]\) has \(q\) components, then there is a connected
dominating set \(X\supseteq D\) such that
\[
 |X\setminus D|\leq 2(q-1).
\]
\end{lemma}

\begin{proof}
Let \(\mathcal C\) be the set of components of \(G[D]\), and form a graph
\(J\) with vertex set \(\mathcal C\) in which two components are adjacent when their
distance in \(G\) is at most three.  We claim that \(J\) is connected.
Indeed, assign to each vertex of \(G \setminus D\) an arbitrary neighbour in \(D\),
and assign each vertex of \(D\) to the component of \(G[D]\) containing it.  Along any path in
\(G\), the components assigned to consecutive vertices are either equal
or have distance at most three.  Since \(G\) is connected, this proves the
claim.

For every edge of a spanning tree of \(J\), take a path of length at most
three between the corresponding components of \(G[D]\), and add its
internal vertices to \(D\).  At most two vertices are added for each of
the \(q-1\) tree edges.  The resulting set is connected, and it remains
dominating because it contains \(D\).
\end{proof}

\begin{lemma}
\label{lem:sampled-components}
Let \(G\) be a graph, and let \(S \subseteq V(G)\) be obtained by sampling every vertex independently
with probability \(p\in[0,1]\).  Then
\[
 \bb{E} \left(\comp(G[S])\right)
 \leq \sum_{v\in V(G)}\frac{1}{\deg_G(v)+1}.
\]
\end{lemma}

\begin{proof}
Choose a uniform random ordering of $V(G)$, independently of $S$. Under this ordering, the first vertex in each component of $G[S]$ precedes all of its neighbours in $S$. Conditional on $S\cap (\{v\}\cup N(v)) \neq \varnothing$, each of the $\deg_G(v)+1$ vertices in $N(v)\cup \{v\}$ is equally likely to be first, so
\begin{align*}
 \bb{E} (\comp(G[S]))  &\leq \sum_{v\in V(G)} \bb{P}(\text{$v$ is the first vertex of $S\cap (N(v)\cup \{v\})$})\\
&= \sum_{v\in V(G)} \frac{\bb{P}(S\cap (N(v)\cup \{v\}) \neq \varnothing)}{\deg_G(v)+1}\\
&\leq \sum_{v\in V(G)} \frac{1}{\deg_G(v)+1}. \qedhere
\end{align*}
\end{proof}

\begin{lemma}\label{lem:low-cost}
There are absolute constants \(C_3\) and \(d_{\alpha}\) with the following
property. Let \(d\geq d_{\alpha}\) and \(K\geq1\), and let \(G\) be a graph such that
\[
 \d(G)\geq d,\qquad \Delta(G)\leq Kd,
 \qquad d\geq C_3(K+\log d).
\]
Then there is a non-empty connected set \(X\subseteq V(G)\) such that
\(G[N_G(X)]\) contains a subgraph \(H\) satisfying
\[
 \d(H)\geq d-C_3(K+\log d).
\]
\end{lemma}

\begin{proof}
Let $Q$ be chosen as an induced subgraph of \(G\) with average degree at
least \(d\) and as few vertices as possible. Then
\(Q\) is connected.  Let \(n=\v(Q)\), and write $\e(Q)=\frac{dn}{2}+s$ with $s\geq0$.
By the minimality of \(Q\), every \(v\in V(Q)\) has degree satisfying 
\[\frac{2(\e(Q)-\deg_Q(v))}{n-1}<d.\] 
Hence, $\deg_Q(v)>\frac d2+s$, and we get $s<\frac{dn}{2(n-2)}$ by summing over $v$. Taking \(d_{\alpha}\) to be large enough so that \(n\geq4\) and using the bound on $s$ gives $\delta(Q)>\frac d2$, $\d(Q)\leq2d$, and $\Delta(Q)\leq Kd$.

Now let \(p=20\log d/d\) and assume that $d_{\alpha}$ is sufficiently large so that \(p\leq1\). Choose a random subset
\(S\subseteq V(Q)\) by sampling every vertex independently with
probability \(p\). Let $U=\{v\in V(Q):N_Q(v)\cap S=\varnothing\}$. Importantly, the set \(D:=S\cup U\) dominates \(Q\). 

Next, we turn to connectedness. Applying \cref{lem:connect-dominating} to \(D\) in \(Q\) produces a
connected dominating set \(X\supseteq D\) by adding at most
\(2(\comp(Q[S])+|U|-1)\) vertices.  Consequently,
\begin{equation}\label{e:lowcost3}
 \sum_{x\in X}\deg_Q(x)\leq  \sum_{v\in S\cup U}\deg_Q(v) +2\Delta(Q)(\comp(Q[S])+|U|) :=Z.
\end{equation}
Here, $Z$ is a bound on the total degree cost of constructing a connected dominating set, where the first summation is the cost of the initial dominating set $D$, and the second term is an upper bound on the cost of connecting $D$. 

We now bound the expected value of $Z$.
By \cref{lem:sampled-components} we know that $\bb{E}(\comp(Q[S]))\leq\frac{2n}{d}$. Moreover, for every \(v\) we have $\bb{P}(v\in U)
 =(1-p)^{\deg_Q(v)}
 \leq \exp(-pd/2)=d^{-10}$
and hence $\bb{E}(|U|)\leq nd^{-10}$. Recalling that $\d(Q)\leq2d$, we can also compute 
\[\bb{E}\Big(\sum_{v\in S}\deg_Q(v)\Big) = p n\d(Q)\leq \frac{20\log d}{d}  n 2d = 40n\log d\]
and 
\[
\bb{E}\Big(\sum_{v\in U}\deg_Q(v)\Big) = \sum_{v\in V(Q)}\deg_Q(v)\bb{P}(v\in U) \leq d^{-10}n\d(Q)\leq2nd^{-9}.
\]

Combining the above expectations (and using $\Delta(Q)\leq Kd$) gives
\begin{align*}
\bb E(Z)
\leq 40n\log d+2nd^{-9}
    +2Kd\left(\frac{2n}{d}+nd^{-10}\right)\leq Cn(K+\log d)
\end{align*}
for an absolute constant $C$. Hence, we may fix an outcome for which 
\[
\sum_{x\in X}\deg_Q(x) \leq Cn(K+\log d).
\]

Set \(R=V(Q)\setminus X\).  Since \(X\) dominates \(Q\), we have
\(R\subseteq N_Q(X)\), and 
\[
 \e(Q[R])
 \geq \e(Q)-\sum_{x\in X}\deg_Q(x)
 \geq \frac{dn}{2}-Cn(K+\log d).
\]
Since $d\geq C_3(K+\log d)$, choosing $C_3>2C$ ensures that $\e(Q[R])>0$. In particular, $R\neq\varnothing$, and, since $|R|\leq n$,
\[
\d(Q[R])
=\frac{2\e(Q[R])}{|R|}
\geq \frac{2\e(Q[R])}{n}
\geq d-2C(K+\log d)
\]
so we may take $H=Q[R]$.
\end{proof}

\begin{corollary}
\label{cor:almost-regular-model}
There is an absolute constant \(C_4\) such that the following holds.
Let \(r\geq1\), \(K\geq1\), and \(d\geq d_{\alpha}\).  If $\d(G)\geq d$, $\Delta(G)\leq Kd$, and $d\geq C_4r(K+\log d)$, 
then \(G\) contains a dominating \(K_r\)-model.
\end{corollary}

\begin{proof}
The result is immediate for \(r=1\).  Starting with \(G_0=G\), apply
\cref{lem:low-cost} iteratively. More precisely, put $d_i=d-Ci(K+\log d)$,
where \(C\) is a sufficiently large absolute constant.  As long as
\(i<r\), apply the lemma to \(G_{i-1}\), with lower average-degree
parameter \(d_{i-1}\) and almost-regularity parameter at most \(2K\).
This produces a
non-empty connected set \(X_i\subseteq V(G_{i-1})\) and a subgraph $G_i\subseteq G_{i-1}[N_{G_{i-1}}(X_i)]$.
Taking \(C_4\) to be sufficiently large, in each of the first $r-1$ steps we maintain $\d(G_i)\geq d_i\geq d/2$ and $\Delta(G_i)\leq Kd\leq2K\d(G_i)$, so we can apply \cref{lem:low-cost} again. By choosing any vertex of \(G_{r-1}\) to be \(X_r\), starting with \(X_r\), and repeatedly using \cref{obs:prepend}, we extract a dominating \(K_r\)-model \((X_1,\ldots,X_r)\).
\end{proof}

\proofsection{The highly unbalanced bipartite case}\label{sec:unbalanced}

We now turn to the case when outcome (i) of \cref{lem:dichotomy} holds. 
Like in the almost-regular case, our strategy is to take a random dominating set and then join together components to find a connected subgraph whose neighbourhood could then serve as the next bag of the model. However, in order to iterate, we need to ensure that we are passing to another highly unbalanced bipartite graph by controlling the change in parameters each time components are merged. This is the purpose of the following lemma, which will play a role similar to Lemma~\ref{lem:connect-dominating}. The parameter $\sigma$ controls the allowable cost of a merge. 

\begin{lemma}\label{lem:component-localisation}
Let \(R\) be a graph with vertex partition $(S,A,B)$ such that \(F:=R[A\cup B]\) is bipartite with parts $A$ and $B$.
Suppose that \(S\) dominates \(A\cup B\), $A \neq \emptyset$ 
and every vertex of \(A\) has degree exactly \(\delta\geq1\) in
\(F\). Let \(\sigma\geq1\), and let 
\begin{equation}\label{eq:Mdef}
M:=|A|-\sigma(\comp(R[S])-1)-\frac{\e(F)}{\sigma}.
\end{equation}
There is a non-empty connected set \(X\subseteq V(R)\) together with sets
\(A'\subseteq A\), \(B'\subseteq B\), $B' \neq \emptyset$ such that $A'\cup B'\subseteq N_R(X)$, $\deg_{F[A',B']}(a)=\delta$ for every $a\in A'$,
and
\begin{equation}\label{eq:localisation-ratio}
 \frac{|A'|}{|B'|}
 \geq
 \frac{M}{|B|+\e(F)/\sigma}.
\end{equation}
\end{lemma}

\begin{proof}
Let \(\mc{C}\) be the set of components of \(R[S]\).
Our strategy is to merge these components by moving particular vertices from $F$ to $S$ according to the procedure that we describe below. During this process, we will refer to the \defn{current} $S$, $F$, $A$, and $B$ to refer to the modified graphs mid-process, and similarly use \(\deg_F\) and \(N_F\) for degrees and neighbourhoods in the current graph.
 For a \defn{remaining} vertex
\(v\in A\cup B\) (that is, a vertex $v$ that has not been moved to $S$ in any previous step), let \(t(v)\) be the number of members of \(\mc{C}\) adjacent
to \(v\). Since the original \(S\) dominates \(A\cup B\), and components can
only merge as \(S\) grows, \(t(v)\geq1\) throughout.

The merging procedure is carried out in two phases, followed by an additional step to smooth out degrees. First, while some remaining \(a\in A\) satisfies \(t(a)\geq2\), move \(a\) to
\(S\). This operation reduces the number of components by at least
\(t(a)-1\geq1\), while costing one vertex. When this phase ends, every remaining \(a\in A\) is adjacent
to a unique component of \(R[S]\). Note that this remains true throughout later steps.

Next, for a remaining \(b\in B\), let \(m(b)\) be the number of components in $\mc{C}$ which are
not adjacent to \(b\) but are the unique adjacent component of at least one
vertex of \(N_F(b)\). While there is a remaining \(b\) that satisfies $\deg_F(b)\leq\sigma\bigl(t(b)+m(b)-1\bigr)$, move $\{b\}\cup N_{F}(b)$ to \(S\). This decreases the number of components by at least
\(t(b)+m(b)-1\) components, so at most \(\sigma\) vertices of $A$ are moved to $S$ per component eliminated in this phase. Consequently, the first two phases move at most $\sigma(\comp(R[S])-1)$
vertices from \(A\) to $S$ in total.

At the conclusion of the second phase, $t(b)+m(b)-1<\deg_F(b)/\sigma$ for every remaining $b\in B$. For every pair \((b,C)\) counted by \(m(b)\), choose one \defn{representative}
\(a_{b,C}\in N_F(b)\) whose unique adjacent component is \(C\). The number of representatives is at most $
\sum_bm(b)
 \leq\frac{1}{\sigma}\sum_b\deg_F(b)
 \leq\frac{\e(F)}{\sigma}.$
The last step of the procedure is to move all
chosen representatives to \(S\) simultaneously. Note that this does not change the number of components. Let the \defn{old} components be those before the third step, while \defn{final} objects are those after representatives are moved.

Let \(A^*,B^*\) be the final remaining sides, and let \(X_1,\ldots,X_s\) be the components of the final \(R[S]\). For each $i$, define $A_i=N_R(X_i)\cap A^*$ and $B_i=N_R(X_i)\cap B^*$. Observe that the sets \(A_i\) partition \(A^*\) since each vertex of \(A^*\) is still adjacent to a unique \(X_i\). If an edge
\(ab\in E(F[A^*,B^*])\) has \(a\in A_i\), then either \(b\) was already
adjacent to \(X_i\) before the last step, or the pair \((b,X_i)\)
was counted by \(m(b)\) and its chosen representative made \(b\) adjacent to
\(X_i\). Hence all of the \(\delta\) original \(F\)-neighbours of a
vertex in \(A_i\) belong to \(B_i\) (the number of neighbours does not drop since each time a vertex from $b$ was moved to $S$, all of its neighbours in $A$ were also moved). This means that if $a\in A_i$, we have
\begin{equation}\label{eq:degaibi}
\deg_{F[A_i,B_i]}(a)=\delta.
\end{equation}

For a vertex $b\in B^*$, let \(\tau(b)\) be the number of final components adjacent to $b$ (with $t(b)$ its old counterpart). Then \(\tau(b)\leq t(b)+m(b)\) since any representative adjacent to $b$ is adjacent to a unique old component. It follows that $\tau(b)-1<\deg_F(b)/\sigma$, and then 
\begin{equation}\label{eq:localisation-b-total}
 \sum_{i=1}^s|B_i|
 =\sum_{b\in B^*}\tau(b)
 \leq |B|+\frac{\e(F)}{\sigma}.
\end{equation}
At the same time, putting together our bounds on the number of vertices moved from $A$ to $S$ each part of the procedure, we get
\begin{equation}\label{eq:localisation-a-total}
 \sum_{i=1}^s|A_i|=|A^*|
 \geq |A|-\sigma(\comp(R[S])-1)-\frac{\e(F)}{\sigma}=M.
\end{equation}
If $M \leq 0$ the lemma trivially holds.
Otherwise, some \(A_i\) is non-empty so \(B_i\) is also non-empty from \eqref{eq:degaibi} (since $\delta\geq 1)$, and averaging \eqref{eq:localisation-b-total} and
\eqref{eq:localisation-a-total} gives an index \(i\) for which
\eqref{eq:localisation-ratio} holds. We can then take \(X=X_i\), \(A'=A_i\), and
\(B'=B_i\). The set \(X\) is non-empty and connected, and the degree condition follows from \eqref{eq:degaibi}.
\end{proof}

In our overall process of extracting a dominating model, it will not always be possible to simply iterate the preceding lemma. Instead, the next lemma shows that in the highly unbalanced setting, either we can find an almost-regular subgraph in which case we exit and conclude by the earlier arguments, or we can find a subgraph to which we will apply \cref{lem:component-localisation} to perform one iteration in the construction of the dominating model.

\begin{lemma}\label{lem:sparse-dominating-seed}
There are absolute constants \(C_0,d_{\beta}>0\) such that the following holds.
Let \(d\geq d_{\beta}\) be an integer and let \(Q\) be bipartite with bipartition
\((A,B)\), where $B \neq \emptyset$. Let 
$r=\frac{|A|}{|B|}$. Suppose that $\deg_Q(a)=d$ for every $a\in A$, $r \geq d^{100}$, and $\deg_Q(b)\geq r$ for every $b\in B$
Then at least one of the following holds.
\begin{enumerate}[label=\textup{(\alph*)}]
 \item\label{item:seed-regular}
 \(Q\) contains a subgraph \(H\) such that $\d(H)\geq\frac{d}{C_0\log d}$ and $\Delta(H)\leq C_0\log d\,\d(H)$.

 \item\label{item:seed-structured}
 There are pairwise disjoint sets \(S_0,A_0,B_0\), with
 \(S_0\subseteq V(Q)\), \(A_0\subseteq A\), and \(B_0\subseteq B\), and a
 bipartite graph \(F\subseteq Q[A_0,B_0]\) such that
\begin{enumerate}[label={(\roman*)}, nosep]
  \item $S_0\text{ dominates }A_0\cup B_0\text{ in }Q$, \label{eq:seed-dominates}
  \item $\comp(Q[S_0])\leq |B|d^{-38}$, \label{eq:seed-components}
  \item $|A_0|\geq(1-C_0d^{-20})|A|$, and \label{eq:seed-a-size}
  \item there is an integer $\delta$ such that $\deg_F(a)=\delta\geq d-C_0\log d$ for all $a\in A_0$ \label{eq:seed-degree}
\end{enumerate}
\end{enumerate}
\end{lemma}

\begin{proof}
Set $q=d^{-20}$ and $p=\frac{10^5\log d}{d}$. By taking \(d_{\beta}\) sufficiently large, we may assume that \(p\leq1/2\). Choose \(A^\ast\subseteq A\) by randomly sampling vertices independently with probability
\(q\). For \(b\in B\), let $\mu_b=q\deg_Q(b)$ be the expected degree of $b$ to $A^*$, and let $Z_b=N_Q\bigl(N_Q(b)\cap A^\ast\bigr)\subseteq B$ be the double neighbourhood. Call \(b\) \defn{thin} if
\(|N_Q(b)\cap A^\ast|<\mu_b/2\). From the assumptions we have
\(\mu_b\geq d^{80}\), so $\bb{P}(b\text{ is thin})\leq\exp(-d^{80}/8)$ by \cref{lem:chernoff}.

Assume that \ref{item:seed-regular} fails.

\begin{claim}\label{claim:seedlarge}
For every choice of \(A^\ast\) and every vertex \(b\) which is non-thin relative to that choice, we have $|Z_b|>d^{40}$.
\end{claim}
\begin{poc}
If this does not hold then choose a set \(Y\subseteq N_Q(b)\cap A^\ast\) of size
\(d^{41}\); this is possible since a non-thin \(b\) has at least
\(d^{80}/8\) neighbours in \(A^\ast\) for large \(d\). Every neighbour of
every \(a\in Y\) lies in \(Z_b\), and hence $\e(Q[Y,Z_b])=d|Y|$ and $\d(Q[Y,Z_b])
 =\frac{2d|Y|}{|Y|+|Z_b|}\geq d$, while \(\Delta(Q[Y,Z_b])\leq d^{41}\). Applying
\cref{lem:regularise} with \(L=d^{40}\) gives
\ref{item:seed-regular}, a contradiction.
\end{poc}

Next choose a random subset \(B^\ast\subseteq B\) by sampling each vertex independently with probability \(p\). A
non-thin vertex \(b\) is called \defn{bad} if $|Z_b\cap B^\ast|<\frac p2|Z_b|$.
Conditional on \(A^\ast\), we can use Chernoff's bound and \cref{claim:seedlarge} to compute that $\bb{P}(b\text{ is bad}\mid A^\ast)
 \leq\exp(-pd^{40}/8)$.

Let \(E_B\) be the set of vertices that are thin or bad, and let 
\[
E_A=\{a\in A:\deg_{B^\ast}(a)\notin[pd/2,2pd]\}.\]
Since \(\deg_{B^\ast}(a)\) is binomial with mean
\(pd=10^5\log d\), the previous probabilities and Chernoff's bound imply, for all sufficiently large \(d\), that $\bb{E}(|E_A|)\leq |A|d^{-1000}$ and $\bb{E}(|E_B|)\leq |B|d^{-1000}$.
For each fixed \(a\in A\), a union bound over its \(d\) neighbours gives $\bb{P}\bigl(a\in N_Q(E_B)\bigr)\leq d^{-999}$. Together with \(\bb{E}(|A^\ast|)=q|A|\), Markov's inequality shows that there is an
outcome that simultaneously satisfies
\begin{enumerate}[label={(\arabic*)}, nosep]
  \item $|A^\ast|\leq8q|A|$, \label{eq:seed-a-star-size}
  \item $|E_A|\leq d^{-100}|A|$, \label{eq:seed-ea-size}
  \item $|E_B|\leq d^{-100}|B|$, and \label{eq:seed-eb-size}
  \item $|N_Q(E_B)\cap A|\leq d^{-100}|A|$. \label{eq:seed-eb-neighbourhood-size}
\end{enumerate}
Indeed, the failure probability of \ref{eq:seed-a-star-size} is at most
\(1/8\), and the sum of the other three failure probabilities is at most
\(3d^{-899}\). 

Fix such an outcome and define
\begin{align*}
 S_0&=\{a\in A^\ast:N_Q(a)\cap B^\ast\neq\varnothing\}\cup B^\ast,\\
 A_0&=A\setminus\bigl(A^\ast\cup E_A\cup N_Q(E_B)\bigr),\\
 B_0&=B\setminus(B^\ast\cup E_B).
\end{align*}
It remains to show that these choices satify the required conditions.
Every \(a\in A_0\) lies outside \(E_A\), and hence has at least
\(pd/2\geq1\) neighbours in \(B^\ast\). If \(b\in B_0\), then \(b\) is
non-thin and non-bad, so there exist \(z\in Z_b\cap B^\ast\) and
\(a\in N_Q(b)\cap A^\ast\) with \(az\in E(Q)\). Thus, $a\in A^\ast$, $N_Q(a)\cap B^\ast\neq\varnothing$, and \(ab\in E(Q)\). This proves
\ref{eq:seed-dominates}.

Next, note that every component of \(Q[S_0]\) contains a vertex of \(B^\ast\). A component containing some \(b\in B^\ast\setminus E_B\) contains every vertex of
\(Z_b\cap B^\ast\) because each such vertex is joined to \(b\) by a path of the form baz inside \(Q[S_0]\). By \cref{claim:seedlarge} and the definition of
badness, it therefore contains more than \(pd^{40}/2\) vertices of
\(B^\ast\). Every other component contains a distinct vertex of
\(B^\ast\cap E_B\). Consequently, using \ref{eq:seed-eb-size},
\[
 \comp(Q[S_0])
 \leq |E_B|+\frac{2|B^\ast|}{pd^{40}}
 \leq d^{-100}|B|+\frac{2|B|}{10^5d^{39}\log d}
 \leq |B|d^{-38}.
\]
This proves \ref{eq:seed-components}. Making $C_0$ sufficienty large and using \ref{eq:seed-a-star-size}, \ref{eq:seed-ea-size}, and
\ref{eq:seed-eb-neighbourhood-size} also gives \ref{eq:seed-a-size}.

Finally, a vertex \(a\in A_0\) has no neighbour in \(E_B\) and at most
\(2pd\) neighbours in \(B^\ast\). It therefore has at least \(d-2pd\)
neighbours in \(B_0\). Set
\[
 \delta=\lfloor d-2pd\rfloor
\]
and, for every \(a\in A_0\), arbitrarily delete all but \(\delta\) of its edges to
\(B_0\). The resulting graph \(F\) satisfies \ref{eq:seed-degree} after making \(C_0\) sufficiently large.
\end{proof}

The next lemma is the main technical result of this section. It is derived by combining Lemmas~\ref{lem:component-localisation} and~\ref{lem:sparse-dominating-seed}.
\begin{lemma}\label{lem:one-step}
There are absolute constants \(C_5\) and \(d_{\gamma}\) such that the following
holds for every integer \(d\geq d_{\gamma}\).  Let \(Q\) be a bipartite graph
with bipartition \((A,B)\), assume that every vertex of \(A\) has degree
exactly \(d\).  If \(|A|/|B|\geq d^{100}\), then at
least one of the following holds.
\begin{enumerate}[label=\textup{(\roman*)}]
 \item\label{item:step-regular}
 \(Q\) contains a subgraph \(H\) such that $\d(H)\geq\frac{d}{C_5\log d}$ and $\Delta(H)\leq C_5\log d\,\d(H)$;
 \item\label{item:step-absorb}
 there are a non-empty connected set \(X\subseteq V(Q)\) and a
 bipartite subgraph \(Q'\subseteq Q[N_Q(X)]\) with parts $A'$, $B'$ such that every vertex of \(A'\) has the same degree
 \(d'\), where $d'\geq d-C_5\log d$ and $\frac{|A'|}{|B'|}
  \geq \frac{|A|}{|B|}(1-d^{-12})$.
\end{enumerate}
\end{lemma}

\begin{proof} Assume without loss of generality that the lemma holds for all proper subgraphs of $Q$ satisfying the conditions of the lemma. Let $r = |A|/|B|$.
	If some
\(b\in B\) has fewer than \(r\) neighbours, delete \(b\) together with all
its neighbours. Then 
\[
 \frac{|A|-\deg_Q(b)}{|B|-1} \geq \frac{r|B|-r}{|B|-1}  =r \geq d^{100}.
\]
As the lemma holds for $Q \setminus \{b\} \setminus N_Q(b)$, by our assumption it also holds for $Q$. Thus, we assume $\deg_Q(b)\geq r$ for every $b\in B$.

Assume that \(d_{\gamma}\geq d_{\beta}\) and apply \cref{lem:sparse-dominating-seed}. If
\cref{lem:sparse-dominating-seed}\ref{item:seed-regular} holds, then we have \ref{item:step-regular} by taking a large enough \(C_5\). Otherwise, we are in the case \cref{lem:sparse-dominating-seed}\ref{item:seed-structured} and we can take
\(S_0,A_0,B_0,F,\delta\) as in that statement. 

Form the auxiliary graph \(R\) on
\(S_0\cup A_0\cup B_0\) whose edges are all edges of \(Q[S_0]\), all edges of
\(Q\) between \(S_0\) and \(A_0\cup B_0\), and all edges of \(F\). Then $R[A_0\cup B_0]=F$, and \(S_0\) dominates \(A_0\cup B_0\) in \(R\). Set $k=\comp(R[S_0])=\comp(Q[S_0])$, $m=\e(F)=\delta|A_0|$, and $\sigma=rd^{20}$.

By \cref{lem:sparse-dominating-seed}, we have $|A_0|\geq(1-C_0d^{-20})|A|$, $k\leq|B|d^{-38}$, and $m\leq d|A|$. Since \(|A|=r|B|\), the latter two inequalities give $\sigma k \leq rd^{20}|B|d^{-38}=|A|d^{-18}$ and $\frac{m}{\sigma}
 \leq\frac{d|A|}{rd^{20}}=|B|d^{-19} \leq|A|d^{-119}$. It follows that the quantity \(M\) defined in \eqref{eq:Mdef} can be bounded as
 \[
M \geq |A|(1-C_0d^{-20}-d^{-18} - d^{-119})
 \]
 which is positive for large \(d\). Now applying
\cref{lem:component-localisation} to \(R,S_0,A_0,B_0,F\) and \(\sigma\) gives a connected set \(X\) and sets
\(A',B'\subseteq N_R(X)\subseteq N_Q(X)\) such that every vertex of \(A'\) has
degree \(\delta\) in \(F[A',B']\) and
\begin{align*}
 \frac{|A'|}{|B'|} \geq
 \frac{|A_0|-\sigma(k-1)-m/\sigma}{|B_0|+m/\sigma}\geq
 r\,\frac{1-Cd^{-18}}{1+d^{-19}}\geq r(1-d^{-12}),
 \end{align*}
where the middle inequality holds for \(d_{\gamma}\) sufficiently large. Taking
\(Q'=F[A',B']\) and
\(d'=\delta\geq d-C_0\log d\geq d-C_5\log d\) proves that
\ref{item:step-absorb} holds.
\end{proof}

The next statement is a corollary of \cref{cor:almost-regular-model} and \cref{lem:one-step}, and establishes \cref{thm:degmain} in the unbalanced case.

\begin{corollary}
\label{prop:unbalanced-model}
There are absolute constants \(C_6\) and \(d_{\delta}\) such that the following
holds.  Let \(t\geq1\) and let \(d\geq d_{\delta}\) be an integer satisfying $d\geq C_6t(\log d)^2$.
If \(Q\) is bipartite with bipartition \((A,B)\), every vertex of \(A\)
has degree exactly \(d\), and
 $\frac{|A|}{|B|}\geq d^{110}$,
then \(Q\) contains a dominating \(K_t\)-model.
\end{corollary}

\begin{proof}
The assertion is clear for \(t=1\). To build a larger dominating \(K_t\)-model, we apply \cref{lem:one-step} repeatedly.  If  outcome (ii) occurs in the first $i$ steps, we obtain connected sets \(X_1,\ldots,X_i\) and a bipartite graph
\(Q_i\) with parts $A_i$, $B_i$ such that $\deg_{Q_i}(a) =: d_i\geq d-Ci\log d$ for every $a \in A_i$, and $\frac{|A_i|}{|B_i|}\geq d^{110}\prod_{j<i}(1-d_j^{-12})$.
Since \(t\leq d\), taking \(C_6\) sufficiently large allows us to maintain 
that $d_i\geq d/2$ and $\frac{|A_i|}{|B_i|}\geq \tfrac12d^{110}\geq d_i^{100}$.
This facilitates the next application of \cref{lem:one-step}.

Suppose outcome (i) occurs after $i$
successive applications yielding outcome (ii), and let $s=t-i$.
This produces a graph \(H\subseteq Q_i\) with $h:=\d(H)\geq\frac{d_i}{C_5\log d_i}$ and $\Delta(H)\leq C_5\log d_i\,h$.
Every subgraph of \(Q_i\), including \(H\), has average degree
at most \(2d_i\).  In particular, \(\log h\leq\log(2d)\).
Using the properties of $Q_i$ and the assumption on $d$, we can take \(C_6\) large enough so that $h\geq C_4s\bigl(C_5\log d_i+\log h\bigr)$.
Hence \cref{cor:almost-regular-model} gives a dominating \(K_s\)-model
in \(H\).  Repeated application of \cref{obs:prepend} prepends
\(X_i,\ldots,X_1\) in reverse order and yields the desired
dominating \(K_t\)-model.

If outcome (i) never occurs in the first \(t-1\) steps, we may
choose any vertex of \(Q_{t-1}\) as the final bag and repeatedly apply
\cref{obs:prepend} to build the dominating \(K_t\)-model.
\end{proof}

\proofsection{Proof of the main theorem}\label{sec:final}

Finally, we derive \cref{thm:degmain} from \cref{cor:almost-regular-model} and \cref{prop:unbalanced-model}.

We shall use the elementary fact that, for every fixed \(C'>0\), there
is an absolute constant \(C>0\) for which
\begin{equation}\label{eq:log-comparison}
 x\geq Ct(\log t)^2,\quad t\geq2
 \quad\Longrightarrow\quad
 x\geq C't(\log x)^2.
\end{equation}
Indeed, \(x/(\log x)^2\) is increasing for \(x>\mathrm e^2\), and at
\(x=Ct(\log t)^2\) we have $\log x\leq \log C+3\log t$ for \(t\geq2\). The required conclusion therefore follows on choosing \(C\) large enough
that $C(\log 2)^2
 \geq C'(\log C+3\log 2)^2$.

\degmain*
\begin{proof}
The case \(t=1\) is immediate, so suppose that $t\geq 2$.  Choose a subgraph \(G_0\subseteq G\)
whose average degree $d=\d(G_0)$
is maximum among all subgraphs of \(G\), and let \(n=\v(G_0)\).  Every
subgraph of \(G_0\) has average degree at most \(d\), so \(G_0\) is
\(d\)-degenerate. In particular, every subgraphs \(J \subseteq G_0\)
satisfies \(\e(J)\leq d\,\v(J)\).  By taking $C$ sufficiently large and using
\eqref{eq:log-comparison}, we may assume that
\begin{equation}\label{eq:main-large}
 d\geq C't(\log d)^2
\end{equation}
for any prescribed absolute constant \(C'\), and that $d$ is large enough to satisfy the conditions in all of the earlier statements in this section.

Apply \cref{lem:dichotomy} to \(G_0\) with $L=d^{400}$.
If \cref{lem:dichotomy}\ref{item:dichotomy-regular} holds then we have a subgraph \(H\) which, setting \(h=\d(H)\), satisfies $h\geq\frac{d}{C\log d}$ and $ \Delta(H)\leq C\log d\,h$.
Since \(h\leq d\) by the choice of \(G_0\), we can let \(C'\)
be sufficiently large in \eqref{eq:main-large} so that $h\geq C_4t(C\log d+\log h)$. The desired model can then be found by \cref{cor:almost-regular-model}.

We may therefore assume that the unbalanced outcome \cref{lem:dichotomy}\ref{item:dichotomy-unbalanced} holds. In that case, there
is a partition \(V(G_0)=A\cup B\) such that $|A|\geq\frac L4|B|$ and $\e_{G_0}(A,B)\geq\frac{nd}{8}$.
In particular, \(|B|\leq4n/L\). Delete from \(A\) every vertex with
fewer than \(d/16\) neighbours in \(B\), and call the remaining set
\(A_1\). This deletes at most \(dn/16\) crossing edges, so $\e_{G_0}(A_1,B)\geq\frac{nd}{16}$. Since $G_0$ is \(d\)-degenerate, we have $|A_1|+|B|\geq\frac n{16}$. Given that $L$ is large, we certainly have \(|B|\leq n/32\), so $|A_1|\geq\frac n{32}$ and $\frac{|A_1|}{|B|}\geq\frac L{128}$.

Set \(d^\ast=\lfloor d/16\rfloor\).  For every \(a\in A_1\), arbitrarily delete all but \(d^\ast\) of its edges to \(B\). In the resulting bipartite
graph \(Q\), every vertex in $A_1$ has degree \(d^\ast\), so $\frac{|A_1|}{|B|}
 \geq\frac{d^{400}}{128}\geq(d^\ast)^{110}$.
Ensuring that $C$ is also large enough so that $d^\ast\geq C_6t(\log d^\ast)^2$, we can then apply \cref{prop:unbalanced-model} to find the desired model.
\end{proof}

\section{Improper colourings}\label{sec:deficiency}

The goal of this section is to prove \cref{thm:deficiency}.
The next lemma can be considered as a significantly simplified variant of \cref{prop:unbalanced-model} with a weaker requirement on  degrees but a much stronger imbalance condition.

\begin{lemma}\label{lem:regbip}
Let $t$ be a positive integer and let  $G=(A,B)$ be a bipartite graph such that $|A|\geq (t-1)!|B|$ and $\deg(a) = t-1$ for each $a\in A$.
Then $G$ contains a dominating $K_t$-model.
\end{lemma}
\begin{proof}

We proceed by induction on $t$. When $t = 1$ or $2$ the statement is clear since a single vertex and a single edge are dominating $K_1$- and $K_2$-models respectively. So suppose $t\geq 3$. Without loss of generality we may assume that $G$ is connected. 

For a subset $B' \subseteq B$, define $A^{\geq 2}_{B'}$ to be the set of vertices $a\in A$ with at least two neighbours in $B'$. Let us choose such a subset $B' \subseteq B$ maximising $|B'|$ subject to two conditions: first, that $G[A^{\geq 2}_{B'} \cup B']$ is connected, and second, that $|A^{\geq 2}_{B'}| < (t-1)! \cdot |B'|$.
We then define $A^{1}_{B'}$ to be the set of vertices $a\in A$ with exactly 1 neighbour in $B'$, and let $B'' = N(A^{1}_{B'}) \setminus B'$ be its neighbourhood outside $B'$. 

\begin{claim}
	$A^{1}_{B'}\neq \emptyset$.
\end{claim}
\begin{poc}
Assume that $A^{1}_{B'} = \emptyset $	for a contradiction.
Since $A,B$ satisfy $|A|\geq (t-1)!|B|\geq (t-1)! |B'|$, we must have $A^{\geq 2}_{B'} \neq A$ and $B' \neq B$. Let $Z= A^{\geq 2}_{B'} \cup B'$  As $G$ is connected there is an edge joining a vertex of $Z$ to  a vertex in  $V(G)-Z$. If this edge has an end in $B'$ then its second end is in $A^{1}_{B'}$, a contradiction. So we assume this edge has ends $a \in A^{\geq 2}_{B'} $ and $b \in B \setminus B'$. Adding $b$ to $B'$ now contradicts maximality of $B'$.
\end{poc}

As every vertex in $A^{1}_{B'}$ has one neighbour in $B'$ and $(t-2)$ in $B''$ and $t \geq 3$, the above claim implies $B'' \neq \emptyset.$

\begin{claim}
If $|A^{1}_{B'}| \geq ((t-2)!+ 1) \cdot |B''|$, then $G$ contains a dominating $K_t$-model.
\end{claim}
\begin{poc}
Choose a minimal set of vertices $S \subseteq A^{1}_{B'}$ such that for every $b\in B''$ there is a neighbour of $b$ in $S$. Note that $|S|\leq |B''|$. Define $A' := A^{\geq 2}_{B'} \cup S$, and $A'':= A^{1}_{B'} \setminus S$. As $B'$ dominates $A^{1}_{B'}$ and hence $A''$ by construction, and $A'$ dominates $B''$, we have $A' \cup B'$ dominates $A'' \cup B''$. In addition, $G[A' \cup B']$ is connected since $G[A^{\geq 2}_{B'} \cup B']$ was connected and all vertices in $S$ have a neighbour in $B'$. We will use $A' \cup B'$ as the first set in our dominating model.

To find the rest of the model, observe that each vertex in $A''$ has exactly $t-2$ neighbours in $B''$, and $$|A''| = |A^{1}_{B'}| - |S|\geq |A^{1}_{B'}| - |B''| \geq (t-2)!\cdot |B''|$$ using the assumption of the claim. By the induction hypothesis, $G[A'' \cup B'']$ contains a dominating $K_{t-1}$-model $(F_1,\dotsc, F_{t-1})$. Then $(A'\cup B', F_1,\dotsc, F_{t-1})$ is a dominating $K_{t}$-model, as desired.
\end{poc}

To conclude the proof, we show that if $|A^{1}_{B'}| < ((t-2)!+ 1) \cdot |B''|$ then $|B'|$ is not maximal to reach a contradiction.
As vertices in $A^{1}_{B'}$ have exactly one neighbour in $B'$, we can bound
\begin{align*}
\e_G(A^{1}_{B'},B'') = (t-2)|A^{1}_{B'}| \leq (t-2)((t-2)! + 1) |B''| \leq (t-1)! |B''|.
\end{align*}
Hence, there is a vertex $b \in B''$ with at most $(t-1)!$ neighbours in $A^{1}_{B'}$. Call this set of neighbours $N'(b) \subseteq A^{1}_{B'}$. Note that $A^{\geq 2}_{B'\cup b}=A^{\geq 2}_{B'} \cup N'(b)$. Since $b$ is adjacent to all of $N'(b)$, and each vertex in $N'(b)$ already had a neighbour in $B'$, it follows that $G[A^{\geq 2}_{B'\cup b} \cup (B'\cup b)]$ is connected. Finally, we have
\begin{align*}
|A^{\geq 2}_{B'\cup b}| = |A^{\geq 2}_{B'}| + |N'(b)| < (t-1)! \cdot |B'| + (t-1)! = (t-1)! \cdot |B'\cup b|.
\end{align*}
Thus, $B'\cup b$ (and the corresponding set $A^{\geq 2}_{B'\cup b}$) satisfies Conditions 1 and 2 above, which contradicts the maximality of $|B'|$.
\end{proof}

We are now ready to prove \cref{thm:deficiency}, which we restate for convenience.
\deficiency*
\begin{proof}
	The case $t=1$ is trivial, so we may assume that $t > 1$.
Suppose $G$ is a minimal counterexample.
Using the constant $C$ from \cref{thm:degmain}, set $D_t = \max(Ct(\log t)^2, t)$, and
\[
c_t=D_t+(t-1)!(D_t-t+1).
\]
Define $B = \{v\in V(G) \colon \deg(v)> c_t\}$ to be the vertices of degree larger than $c_t$ and $A=V(G)\setminus B$ the remaining vertices.
If $B=\emptyset$ then $\Delta(G)\leq c_t$ so $G$ is not a counterexample.

Suppose that $a\in A$ has at most $t-2$ neighbours in $B$. Since $G$ is a minimal counterexample, the graph $G\setminus \{a\}$ may be partitioned into $t-1$ parts $X_1', \dotsc, X_{t-1}'$ such that  $\Delta(G[X'_i]) \leq c_t$.
If there were an $i$ such that $a$ has no neighbours in $X'_{i} \cap B$, then replacing $X'_i$ by $X'_i \cup \{a\}$ would give a partition of $V(G)$ with each part inducing a graph of maximum degree at most $c_t$. Hence we may assume that each $a\in A$ has at least $t-1$ neighbours in $B$. 

As $G$ is a counterexample, \cref{thm:degmain} gives $\d(G) <  D_t$.
In particular, $c_t|B| + (t-1)|A| < D_t(|A|+|B|)$, and hence
\[
|A| \geq \frac{c_t-D_t}{D_t-t+1}|B|=(t-1)!|B|.
\]
By removing edges from $G$, we can obtain a bipartite subgraph with bipartition $(A,B)$ that satisfies the conditions of \cref{lem:regbip}. This implies that $G$ contains a dominating $K_t$-model, a contradiction.
\end{proof}

\section{Concluding remarks}\label{sec:remarks}

The main result of this paper proves an upper bound of $O(t (\log t)^2)$ on the average degree sufficient to guarantee a dominating $K_t$-model. A logarithmic gap between this bound and the known lower bound $\Omega(t\log t)$  remains. It would be interesting to close this gap. We suspect that the lower bound is correct, but our methods lose an additional $\log t$  factor by separating the almost regular and  highly unbalanced cases.

Our second result shows that the Dominating Hadwiger's conjecture  (\cref{c:DomHad}) holds for improper colouring. As mentioned in the introduction, a more restrictive relaxation of Hadwiger's conjecture to clustered colorings has been established by Dujmovi\'c et al.~\cite{dujmovic2023clustered}. The proof in~\cite{dujmovic2023clustered} uses powerful structural tools which are not available in the dominating setting. Thus it would be interesting (but we expect very challenging!) to prove a similar relaxation of the Dominating Hadwiger's conjecture.

\begin{conjecture}[The Clustered Dominating Hadwiger's conjecture] 
For every  positive integer $t$ there exists $c_t>0$ such that the vertex set of any graph $G$ that does not contain a dominating $K_t$-model can be partitioned into $t-1$ parts $X_1,\dotsc, X_{t-1}$ so that $G[X_i]$ has no component with more than $c_t$ vertices for every $1 \leq i \leq t-1$.
\end{conjecture}

Finally, determining the truth of  \cref{c:DomHad} remains the central open question on dominating models. 

\subsection*{Acknowledgements}

Part of the research in this paper, leading to the revision of \cref{sec:quadratic} to include \cref{q:Bruce}, was carried out at the 2026 Barbados Graph Theory Workshop. We thank the workshop organizers and participants for creating
a stimulating working environment. 

\subsubsection*{AI Disclosure} The main results of this paper were obtained in September 2024 without any AI use. ChatGPT 5.6 Sol and Codex were later used to fill in the details of the proof of \cref{thm:degmain} and reorganize parts of it. The output was checked and edited by the authors, who take full responsibility for its correctness. ChatGPT 5.6 Sol was also used for proofreading the paper.

\bibliographystyle{style_edited}
\bibliography{references}

@article{IW25,
  title={{Dominating $K_t$-models}},
  author={Freddie Illingworth and David Wood},
  journal={Journal of Graph Theory},
  volume = {110},
  number = {4},
  pages = {448-456},
  year={2025}
}

@article{edwards2015relative,
  title={A relative of {H}adwiger's conjecture},
  author={Edwards, Katherine and Kang, Dong Yeap and Kim, Jaehoon and Oum, Sang-il and Seymour, Paul},
  journal={SIAM Journal on Discrete Mathematics},
  volume={29},
  number={4},
  pages={2385--2388},
  year={2015},
  publisher={SIAM}
}

@article{van2018improper,
  title={Improper colourings inspired by {H}adwiger's conjecture},
  author={van den Heuvel, Jan and Wood, David R},
  journal={Journal of the London Mathematical Society},
  volume={98},
  number={1},
  pages={129--148},
  year={2018},
  publisher={Wiley Online Library}
}

@article{girao2025induced,
  title={Induced subdivisions in ${K_{s,s}}$-free graphs with polynomial average degree},
  author={Gir{\~a}o, Ant{\'o}nio and Hunter, Zach},
  journal={International Mathematics Research Notices},
  volume={2025},
  number={4},
  pages={rnaf025},
  year={2025},
  publisher={Oxford University Press}
}

@book{mitzenmacher2017probability,
	title        = {Probability and computing},
	author       = {Mitzenmacher, Michael and Upfal, Eli},
	year         = 2017,
	publisher    = {Cambridge University Press, Cambridge},
	pages        = {xx+467},
	note         = {Randomization and probabilistic techniques in algorithms and data analysis},
	edition      = {Second},
	mrnumber     = 3674428
}

@article{hadwiger1943klassifikation,
  title={{\"Uber eine Klassifikation der Streckenkomplexe}},
  author={Hadwiger, Hugo},
  journal={Vierteljahrsschrift der Naturforschenden Gesellschaft in Z{\"u}rich},
  volume={88},
  pages={133--143},
  year={1943}
}

@incollection{seymour2016hadwiger,
  title={{Hadwiger's conjecture}},
  author={Seymour, Paul D.},
  booktitle={Open Problems in Mathematics},
  editor={Nash, Jr., John Forbes and Rassias, Michael Th.},
  pages={417--437},
  publisher={Springer},
  address={Cham},
  year={2016},
  doi={10.1007/978-3-319-32162-2_13}
}

@article{kostochka1984lower,
  title={{Lower bound of the Hadwiger number of graphs by their average degree}},
  author={Kostochka, A. V.},
  journal={Combinatorica},
  volume={4},
  number={4},
  pages={307--316},
  year={1984},
  doi={10.1007/BF02579141}
}

@article{thomason1984extremal,
  title={{An extremal function for contractions of graphs}},
  author={Thomason, Andrew},
  journal={Mathematical Proceedings of the Cambridge Philosophical Society},
  volume={95},
  number={2},
  pages={261--265},
  year={1984},
  doi={10.1017/S0305004100061521}
}

@article{fernandez1983maximum,
  title={{On the maximum density of graphs which have no subcontraction to $K_s$}},
  author={Fern{\'a}ndez de la Vega, Wenceslas},
  journal={Discrete Mathematics},
  volume={46},
  number={1},
  pages={109--110},
  year={1983},
  doi={10.1016/0012-365X(83)90280-7}
}

@article{girao2026dominating,
  title={{The Dominating 4-Colour Theorem}},
  author={Gir{\~a}o, Ant{\'o}nio and Illingworth, Freddie and Mohar, Bojan and Norin, Sergey and Steiner, Raphael and Tamitegama, Youri and Tan, Jane and Wood, David R. and Yip, Jung Hon},
  journal={arXiv:2605.10112 \emph{preprint}},
  year={2026},
  doi={10.48550/arXiv.2605.10112}
}

@book{appel1989every,
  title={{Every Planar Map is Four Colorable}},
  author={Appel, Kenneth and Haken, Wolfgang},
  series={Contemporary Mathematics},
  volume={98},
  publisher={American Mathematical Society},
  address={Providence, RI},
  year={1989}
}

@article{robertson1997four,
  title={{The four-colour theorem}},
  author={Robertson, Neil and Sanders, Daniel P. and Seymour, Paul and Thomas, Robin},
  journal={Journal of Combinatorial Theory, Series B},
  volume={70},
  number={1},
  pages={2--44},
  year={1997},
  doi={10.1006/jctb.1997.1750}
}

@article{janzer2023Erdos,
  title={Resolution of the {Erd{\H{o}}s--Sauer} problem on regular subgraphs},
  volume={11},
  doi={10.1017/fmp.2023.19},
  journal={Forum of Mathematics, Pi},
  author={Janzer, Oliver and Sudakov, Benny},
  year={2023},
  pages={e19}
}

@article{kuhn2025disproof,
  title={Disproof of the Odd {H}adwiger Conjecture},
  author={K{\"u}hn, Marcus and Sauermann, Lisa and Steiner, Raphael and Wigderson, Yuval},
  journal={arXiv:2512.20392 \emph{preprint}},
  year={2025},
  doi={10.48550/arXiv.2512.20392}
}

@article{delcourt2025reducing,
  title={Reducing Linear {H}adwiger's Conjecture to Coloring Small Graphs},
  author={Delcourt, Michelle and Postle, Luke},
  journal={Journal of the American Mathematical Society},
  volume={38},
  number={2},
  pages={481--507},
  year={2025},
  doi={10.1090/jams/1047}
}

@inproceedings{dujmovic2023clustered,
  title={{Proof of the Clustered Hadwiger Conjecture}},
  author={Dujmovi{\'c}, Vida and Esperet, Louis and Morin, Pat and Wood, David R.},
  booktitle={2023 IEEE 64th Annual Symposium on Foundations of Computer Science (FOCS)},
  pages={1921--1930},
  year={2023},
  publisher={IEEE},
  doi={10.1109/FOCS57990.2023.00116}
}

@misc{reed2026barbados,
  author={Reed, Bruce},
  title={{Private communication at the 2026 Barbados Graph Theory Workshop}},
  year={2026}
}
\end{document}